\documentclass[11pt,a4paper]{amsart}
\usepackage{amsmath, amssymb,amsthm}
\usepackage{color, hyperref} 
\usepackage[all]{xy}
\usepackage{bbm}
\usepackage{enumerate,enumitem}
\usepackage{comment}
\usepackage{mathrsfs}
\usepackage{graphicx}
\usepackage{tabularx}
\usepackage{tikz-cd} 
\usepackage{MnSymbol} 

\theoremstyle{plain}
\newtheorem{thm}{Theorem}[section]
\newtheorem{lemma}[thm]{Lemma}
\newtheorem{prop}[thm]{Proposition}
\newtheorem{cor}[thm]{Corollary}
\newtheorem{conj}[thm]{Conjecture}

\theoremstyle{definition}

\theoremstyle{remark}
\newtheorem{remark}[thm]{Remark}

\DeclareMathOperator{\Gal}{Gal}
\DeclareMathOperator{\ord}{ord}

\DeclareMathOperator{\cusp}{cusp}

\DeclareMathOperator{\Tr}{Tr}

\DeclareMathOperator{\Ann}{Ann}
\DeclareMathOperator{\Res}{Res}

\DeclareTextCompositeCommand{\"}{OT1}{i}{\"\i}

\newcommand{\Z}{\mathbb Z}
\newcommand{\Q}{\mathbb Q}

\newcommand{\C}{\mathbb C}

\newcommand{\cO}{\mathcal{O}}

\newcommand{\cT}{\mathcal{T}}

\newcommand{\cE}{\mathcal{E}}

\newcommand{\cL}{\mathcal{L}}
\newcommand{\cH}{\mathcal{H}}
\newcommand{\cI}{\mathcal{I}}

\newcommand{\et}{\mathrm{\acute{e}t}}

\newcommand{\p}{\mathfrak p}

\newcommand{\map}{\longrightarrow}
\newcommand{\isom}{\cong}
\newcommand{\embed}{\hookrightarrow}
\newcommand{\surj}{\twoheadrightarrow}

\newcommand{\<}{\langle}   
\renewcommand{\>}{\rangle} 

\newcommand{\ol}{\overline}
\newcommand{\wt}{\widetilde}

\def\mfr{\mathfrak}

\numberwithin{equation}{section}

\begin{document}
\title{On Sharifi's conjecture: exceptional case}
\author{Sheng-Chi Shih and Jun Wang}
\address{Fakultät für Mathematik, Oskar-Morgenstern-Platz 1, A-1090 Wien, Austria.}
\email{sheng-chi.shih@univie.ac.at}
\address{Morningside Center of Mathematics, No. 55, Zhongguancun East Road, Beijing, 100190, China}
\email{jwangmathematics@gmail.com}
\date{\today}

\begin{abstract}
In the present article, we study the conjecture of 
Sharifi on the surjectivity of the map $\varpi_{\theta}$. Here $\theta$ is a primitive even Dirichlet character of conductor $Np$, which is exceptional in the sense of Ohta. After localizing at the prime ideal $\mfr{p}$ of the Iwasawa algebra related to the trivial zero of the Kubota\textendash Leopoldt $p$-adic $L$-function $L_p(s,\theta^{-1}\omega^2)$, we compute the image of $\varpi_{\theta,\mfr{p}}$ in a local Galois cohomology group and prove that it is an isomorphism. Also, we prove that the residual Galois representations associated to the cohomology of modular curves are decomposable after taking the same localization. 
\end{abstract}
\maketitle
\addtocontents{toc}{\setcounter{tocdepth}{1}}
\tableofcontents

\section{Introduction}\label{sec:01}
In this paper, we study a conjecture of Sharifi that refines the Iwasawa main conjecture. To be more precise, let us start by preparing for notation. Let $N$ be a positive integer, and let $p\geq 5$ be a prime number not dividing $N\phi(N)$, where $\phi(N)$ is the order of the group $(\Z/N\Z)^{\times}$. For a positive integer $M$, we denote by $Y_1(M)$ the moduli space over $\Z[1/M]$ of elliptic curves $E$ with an injective homomorphism $\Z/M\Z\embed E$ and denote by $X_1(M)$ its compactification. Also, we set $\zeta_M:=e^{2\pi i/M}$. Following the convention of \cite{FK}, the cohomology of (compact) modular curves is naturally endowed with the action of dual Hecke operators. Set $H:=\varprojlim_r H^1_{\et}(X_1(Np^r)_{/\ol{\Q}},\Z_p)^{\ord}$, where the superscript ``ord" means the ordinary part for the dual Hecke operator $U^*_p$. Let $\mfr{h}^*=\Z_p[T^*_n \mid n\geq 1]$ be the Hecke algebra acting on $H$. We will write the dual Hecke operator $T^*_q$ as $U^*_q$ for all $q|Np$. Let $I^*$ be the Eisenstein ideal in $\mfr{h}^*$  generated by $T^*_{\ell}-\ell\<\ell\>^{-1}-1$ for all primes $\ell\nmid Np$ and by $U^*_q-1$ for all primes $q|Np$. For a $G_{\Q}:=\Gal(\ol{\Q}/\Q)$-module $M$, we denote by $M^+$ (resp.~$M^-$) the submodule of $M$ on which complex conjugation acts via $1$ (resp.~$-1$).

Let
$$
\varpi:H^-(1)\to \varprojlim_r H^2(\Z[\zeta_{Np^r},\tfrac{1}{p}],\Z_p(2))^+:=S
$$
be the $\Lambda:=\Z_p[(\Z/Np\Z)^{\times}]\lsem 1+p\Z_p \rsem$-module homomorphism constructed by Sharifi in \cite[Section~5.3]{Sha} (or see Section~\ref{sec:02} for the definition). It was conjectured by Sharifi (Conjecture~5.8 in \textit{loc.~cit.}) and proved by Fukaya\textendash Kato \cite[Theorem~5.2.3]{FK} that the kernel of $\varpi$ contains the Eisenstein ideal $I^*$.

For a Dirichlet character $\theta$ of conductor $Np$, we denote by $\Z_p[\theta]$ the $\Z_p$-algebra generated by the values of $\theta$, on which $(\Z/Np\Z)^{\times}$ acts via $\theta$. For a $\Z_p[(\Z/Np\Z)^{\times}]$-modular $A$,  we denote by $A_{\theta}:=A\otimes_{\Z_p[(\Z/Np\Z)^{\times}]} \Z_p[\theta]$ the $\theta$-eigenspace of $A$. It follows from the assumption $p\nmid N\phi(N)$ that one has $H^-(1)=\bigoplus_{\theta} H^-_{\theta}(1)$ and $S=\bigoplus_{\theta} S_{\theta}$, where the direct sums run through all even Dirichlet characters of modulus $Np$. The following is a conjecture of  McCallum\textendash Sharifi \cite[Conjecture~5.3]{McCallum-Sharifi} (or see \cite[Conjecture~5.4]{Sha}). 

\begin{conj}\label{conj:Mc_Sha}
Let $\theta$ be a primitive even Dirichlet character of conductor $Np$. Then the homomorphism $\varpi$ induces a surjective  homomorphism of $\Lambda_{\theta}$-modules
$$
\varpi_{\theta}:H^-_{\theta}(1)/I_{\theta}^*H^-_{\theta}(1)\surj S_{\theta}.
$$
\end{conj}


Denote by $\omega$ the Teichm\"{u}ller character. Following \cite{Oht2}, a primitive even Dirichlet character $\theta$ of conductor $Np$ is said to be exceptional if the character $\chi:=\theta\omega^{-1}$ is trivial on $(\Z/p\Z)^{\times}$ and $\chi|_{(\Z/N\Z)^{\times}}(p)=1$.  When the character $\theta$ is not exceptional, a conjecture of Sharifi \cite[Conjecture~5.2 and the remark at the end of Section~5]{Sha} predicts that $\varpi_{\theta}$ is indeed an isomorphism together with the inverse homomorphism $\Upsilon_{\theta}$. Also, he proved that there is a canonical isomorphism of $\Lambda_{\theta}$-modules $S_{\theta}\isom X_{\chi}(1)$ (Lemma~4.11 in \textit{loc.~cit.}), where $X$ is the Galois group of the maximal abelian unramified pro-$p$ extension of $\Q(\zeta_{Np^{\infty}})$. The main conjecture asserts that the characteristic ideal of $X_{\chi}$ is generated by a certain $p$-adic $L$-function. Thus, the conjecture of $\varpi_{\theta}$ being an isomorphism is a refinement of the main conjecture and predicts the $\Lambda_{\theta}$-module structure of $X_{\chi}$. Some partial results on this conjecture were proved by Fukaya\textendash Kato \cite{FK}, Fukaya\textendash Kato\textendash Sharifi \cite{FKS}, and Wake\textendash Wang-Erickson \cite{WE}. Their ideas are to prove that the homomorphism $\Upsilon_{\theta}$ is an isomorphism and is the inverse of $\varpi_{\theta}$. The requirement of $\theta$ being not exceptional is essential to the construction of $\Upsilon_{\theta}$ as in this case, one can apply a work of Ohta  \cite{Oht2} to show that one has a short exact sequence of $\Lambda_{\theta}[G_{\Q}]$-modules
$$
0 \to H_{\theta}^-/I_{\theta}^*H_{\theta}^- \to H_{\theta}/I_{\theta}^*H_{\theta} \to H_{\theta}^+/I_{\theta}^*H_{\theta}^+ \to 0.
$$

When $\theta$ is exceptional, there is no literature discussing the homomorphism $\varpi_{\theta}$. One of the difficulties in this case is that it is not clear whether one of $H_{\theta}^-/I_{\theta}^*H_{\theta}^-$ and $H_{\theta}^+/I_{\theta}^*H_{\theta}^+$ is $G_{\Q}$-stable, and hence, one can not construct $\Upsilon_{\theta}$ following Sharifi's construction.

One of the goals in this paper is to study the homomorphism $\varpi_{\theta}$ when $\theta$ is exceptional. We will prove that it is an isomorphism after taking the localization at a certain height one prime ideal of $\Lambda_{\theta}$ corresponding to the trivial zero of the Kubuta\textendash Leopoldt $p$-adic $L$-function $L_p(s,\theta^{-1}\omega^2)$ without constructing $\Upsilon_{\theta}$ (Theorem~\ref{11}). It is known that the leading coefficient of this $p$-adic $L$-function involves a certain $\cL$-invariant. Since Sharifi's conjecture is a refinement of the main conjecture, it is natural to ask how such an $\cL$-invariant is related to Sharifi's conjecture. This question will be addressed in Theorem~\ref{12} below. Another goal is to study the residual Galois representation attached to $H_{\theta}=H_{\theta}^-\oplus H_{\theta}^+$ after taking the same localization as in Theorem~\ref{11} (Theorem~\ref{13}). An advantage of taking such localization is that the image of $\varpi_{\theta}$ belongs to a certain local Galois cohomology group so that we are able to compute it explicitly.  Another advantage is that the cuspidal Hecke algebra is Gorenstein by a work of Betina\textendash Dimitrov\textendash Pozzi \cite[Theorem~A(i)]{BDP}, which makes the study of the cohomology of modular curves modulo an Eisenstein ideal easier (for example, see Proposition~\ref{source_of_varpi}). These two advantages are essential in our study. For instance, the former and the latter respectively help us to show the surjectivity and the injectivity of $\varpi_{\theta,\mfr{p}}$ in Theorem~\ref{11}.  

In the remaining of the Introduction, we assume that the character $\theta$ is exceptional. We will view $\chi$ as an odd primitive character of conductor $N$ with $\chi(p)=1$. Let $\kappa: \Gal(\Q(\zeta_{Np^{\infty}})/\Q(\zeta_{N}))\to \Z_p^{\times}$ be the $p$-adic cyclotomic character, and let $\gamma$ be a fixed topological generator of $\Gal(\Q(\zeta_{Np^{\infty}})/\Q(\zeta_{Np}))$. Then we can identify $\Lambda_{\theta}$ with $\Z_p[\theta]\lsem T\rsem$ by sending $\kappa(\gamma)$ to $1+T$. Let $\mfr{p}$ be the prime ideal of $\Lambda_{\theta}$ generated by $T+1-\kappa(\gamma)$, which is related to the trivial zero of the Kubota\textendash Leopoldt $p$-adic $L$-function $L_p(s,\theta^{-1}\omega^2)$ (see Section~\ref{sec:32}). We denote by $\Lambda_{\theta,\mfr{p}}$ the localization of $\Lambda_{\theta}$ at $\mfr{p}$ with residue field $k_{\theta,\mfr{p}}:=\Lambda_{\theta,\mfr{p}}/\mfr{p}$, and for a $\Lambda_{\theta}$-module $M$, we set $M_{\mfr{p}}:=M\otimes_{\Lambda_{\theta}} \Lambda_{\theta,\mfr{p}}$. The following is the first main result in this paper.

\begin{thm}\label{11}
Let $\theta$ be a primitive even Dirichlet character of conductor $Np$. Assume $p\geq 5$ with $p\nmid N\phi(N)$ and assume that $\theta$ is exceptional. Then $\varpi_{\theta}$ induces an isomorphism of $k_{\theta,\mfr{p}}$-vector spaces
$$
\varpi_{\theta,\mfr{p}}:H^-_{\theta,\mfr{p}}(1)/I^*_{\theta,\mfr{p}}H^-_{\theta,\mfr{p}}(1) \isom S_{\theta,\mfr{p}}.
$$
\end{thm}

When $\theta$ is not exceptional, assuming Greenberg's conjecture, an analog result of the above theorem has been proved by Wake\textendash Wang-Erickson \cite[Corollary~C]{WE}.  Their idea is first to show that for each height one prime ideal $\mfr{q}$ of $\Lambda_{\theta}$, the map $\Upsilon_{\theta,\mfr{q}}$ is an isomorphism and then, by a work of Fukaya\textendash Kato \cite{FK}, it is the inverse of $\varpi_{\theta,\mfr{q}}$. When $\theta$ is exceptional, their method can be adapted if one considers the localization at a height one prime ideal $\mfr{q}$ of $\Lambda_{\theta}$ other than $\mfr{p}$. In this case (or even without taking any localization), it is difficult to show that $\varpi_{\theta,\mfr{q}}$ is an isomorphism without using $\Upsilon_{\theta,\mfr{q}}$, since it is defined by the cup product of two cyclotomic units and is difficult to be computed.

A crucial point in the proof of Theorem~\ref{11} is that we are able to compute the image of $\varpi_{\theta,\mfr{p}}$ via the local explicit reciprocity law by showing that the image of $\varpi_{\theta,\mfr{p}}$ is contained in a certain local Galois cohomology group. There are three steps in the proof of the theorem: 
\begin{enumerate}
\item Show that both $H^-_{\theta,\mfr{p}}(1)/I^*_{\theta,\mfr{p}}H^-_{\theta,\mfr{p}}(1)$ and $S_{\theta,\mfr{p}}$ are $1$-dimensional $k_{\theta,\mfr{p}}$-vector spaces (Propositions~\ref{Brauer_structure} and \ref{source_of_varpi}). 
\item Show that $(1-U^*_q)\{0,\infty\}_{\theta,DM}$ is in $H_{\theta}^-$ whose image under $\varpi_{\theta}$ in $S_{\theta}$ is the cup product $(q,1-\zeta_{Np^r})_{r\geq 1, \theta}$ for all primes $q|Np$ (Corollary~\ref{image_of_varpi}). Here $\{0,\infty\}$ is the modular symbol attached to the cusps $0$ and $\infty$, and the subscript ``\textit{DM}" means the Drinfield\textendash Manin modification (see Section~\ref{sec:32} for the definition). 

\item Show that the cup product $(q,1-\zeta_{Np^r})_{r,\theta,\mfr{p}}$ is non-trivial for all $r$ big enough, which is the second main result in this article (Theorem~\ref{12}).
\end{enumerate}

We now introduce notation to state the second main result. Let $\tau(\chi^{-1})$ be the Gauss sum attached to $\chi^{-1}$. Denote by $\cL(\chi)$ the $\cL$-invariant attached to $\chi$ (for example, see \cite[(15)]{BDP} for the definition). It was proved by Gross \cite[Section~4]{gross} (or see \cite[Section~5.3]{BDP} for a different proof) that the $\cL$-invariant $\cL(\chi)$ is related to the derivative of the $p$-adic $L$-function $L_p(s,\chi\omega)$ at $s=0$, namely, one has $L'_p(0,\chi\omega)=\cL(\chi)L(0,\chi)$. 
Moreover, it was proved by Ferrero\textendash Greenberg \cite[Proposition~1]{FG} that $L'_p(0,\chi\omega)$ is a multiple of $p$.

\begin{thm}[Corollary~\ref{cup_product_map}]\label{12}
Let the notation be as above, and let the assumption be as in Theorem~\ref{11}. For each positive integer $r$, we have
$$
(\ell ,1-\zeta_{Np^r})_{r,\theta,\mfr{p}}=\frac{(p-1)\log_p (\ell)}{p \phi(N)} \omega(N)\tau(\chi^{-1}) L(0,\chi)\in (\cO/p^r\cO)(1) 
$$
for all $\ell |N$ and
$$
(p,1-\zeta_{Np^r})_{r,\theta,\mfr{p}}=\frac{(p-1)}{p \phi(N)} \omega(N)\tau(\chi^{-1}) \cL(\chi) L(0,\chi)\in (\cO/p^r\cO)(1).
$$
In particular, $(q,1-\zeta_{Np^r})_{r,\theta,\mfr{p}}$ is non-trivial for all $r$ big enough and for all $q|Np$.
\end{thm}

We notice that the Gauss sum $\tau(\chi^{-1})$ is a unit in $\Z_p[\theta]$ (for example, see \cite[Lemma~6.1(c)]{Wa}). Then, the right hand side of the two formulas in the above theorem are integral because both $\log_p(\ell)$ and $\cL(\chi)L(0,\chi)$ are divisible by $p$. Thus, they have to be non-trivial in $(\cO/p^r\cO)(1)$ for all $r$ big enough as they are independent of $r$. 

\begin{remark}
The condition of $\theta$ being a primitive character is essential in the above theorems. It is also a crucial condition in the proof of \cite[Proposition~4.4]{BDP}, which is used to show the Gorensteinness of $\mfr{h}^*_{\theta,\mfr{p}}$. Thus, when $\theta$ is imprimitive, it is not known whether $\mfr{h}^*_{\theta,\mfr{p}}$ is Gorenstein, and hence, we can not assert that $\varpi_{\theta,\mfr{p}}$ is injective. In addition, for $q|Np$, the image of $(1-U^*_{q})\{0,\infty\}_{\theta,DM,\mfr{p}}$ under $\varpi_{\theta,\mfr{p}}$ may be trivial since $L(0,\chi)$ may be zero when $\chi$ is imprimitive.  
\end{remark}

By the second step of the proof of Theorem~\ref{11} mentioned above, one can show that $H^-_{\theta,\mfr{p}}/I^*_{\theta,\mfr{p}}H^-_{\theta,\mfr{p}}$ is $G_{\Q}$-stable by using a $\Lambda_{\theta}$-adic perfect pairing on $H_{\theta}$ constructed by Ohta (see Section~\ref{sec:05}). We will show that $H^+_{\theta,\mfr{p}}/I^*_{\theta,\mfr{p}}H^+_{\theta,\mfr{p}}$ is also $G_{\Q}$-stable.

\begin{thm}[Theorem~\ref{split_gal_rep_at_p}]\label{13}
Let the notation be as above, and let the assumption be as in Theorem~\ref{11}. Then the following short exact sequence of $k_{\theta,\mfr{p}}[G_{\Q}]$-modules splits
$$
0\to H^-_{\theta,\mfr{p}}/I^*_{\theta,\mfr{p}}H^-_{\theta,\mfr{p}}\to H_{\theta,\mfr{p}}/I^*_{\theta,\mfr{p}}H_{\theta,\mfr{p}}\to 
H^+_{\theta,\mfr{p}}/I^*_{\theta,\mfr{p}}H^+_{\theta,\mfr{p}}\to 0.
$$
\end{thm}

If one considers the localization at a height one prime ideal of $\Lambda_{\theta}$ other than $\mfr{p}$ (or the case that $\theta$ is not exceptional), it was shown by Ohta \cite[Section~3.3]{Oht3} (see \cite[Section~5.3]{Oht1} for not exceptional case) that the corresponding short exact sequence of $G_{\Q}$-modules in Theorem~\ref{13} does not split and the associated ordinary Galois representation is $p$-distinguished. Thus, one can construct an analog of the homomorphism $\Upsilon_{\theta}$ following Sharifi's construction. This construction can not be adapted to the situation of Theorem~\ref{13} since it follows from the assumption $\chi(p)=1$ that the associated Galois representation is not $p$-distinguished.

\subsection{Outline}
In Section~\ref{sec:02}, we briefly review the construction of $\varpi_{\theta}$ following \cite{Sha}. 

In Section~\ref{sec:03}, we first show that $S_{\theta,\mfr{p}}$ is isomorphic to a certain local Galois cohomology group and show that it is a $1$-dimensional $k_{\theta,\mfr{p}}$-vector space (Proposition~\ref{eq:brauer_group}). This phenomena is opposite to the case of $\theta$ being not exceptional. In this case, it is known \cite[Lemma~4.11]{Sha} that $S_{\theta}$ is isomorphic to a certain global Galois cohomology group. Second, we construct some elements in $H_{\theta}^-$ by computing the congruence module attached to \eqref{eq:coho_exact_seq} (Theorem~\ref{DM}). This plays an important role in computing the image of $\varpi_{\theta,\mfr{p}}$. When $\theta$ is not exceptional, it was proved by Ohta \cite[Section~3.5]{Oht2} that the desired congruence module is essentially isomorphic to the congruence module attached to \eqref{eq:modular_forms_seq}. His argument can not be adapted when the character $\theta$ is exceptional since in this case, it is not clear whether the short exact sequences (3.4.6) in \textit{loc.~cit.}~ split as Hecke-modules. Our idea is to use the Drinfeld\textendash Manin modification and the work of Lafferty \cite{Laf} to compute the desired congruence module. Third, we will show that the source of $\varpi_{\theta,\mfr{p}}$ is a $1$-dimensional $k_{\theta,\mfr{p}}$-vector space by a result of Betina\textendash Dimitrov\textendash Pozzi \cite{BDP} on the Gorensteinness of the cuspidal Hecke algebra after localizing at $\mfr{p}$.

Section~\ref{sec:04} is devoted to computing the image of $\varpi_{\theta,\mfr{p}}$ and completing the proof of Theorem~\ref{11}. The main goal of Section~\ref{sec:05} is to prove Theorem~\ref{13}.

\subsection*{Acknowledgments}
The authors would like to thank Adel Betina, Emmanuel Lecouturier, Romyar Sharifi, and  Preston Wake for helpful comments. The authors are grateful to the referee for the careful reading
and suggestions on the improvement of this manuscript.

The first author has been supported by the Labex CEMPI under Grant No.~ANR-11-LABX-0007-01, by I-SITE ULNE under Grant No.~ANR-16-IDEX-0004, and by Austrian Science Fund (FWF) under Grant No.~START-Prize Y966. The second author would like to thank Sujatha Ramdorai and Morningside Center of Mathematics
for supporting his postdoctoral studies. 
\section{Sharifi's Conjecture}\label{sec:02}

Throughout this paper, we keep the notation in the Introduction. The goal of this section is to review the construction of $\varpi_{\theta}$ in Conjecture~\ref{conj:Mc_Sha} following \cite{Sha}. We refer the reader to \textit{loc.~cit.} and \cite[Section~5.2]{FK} for more details.

Set $H_{r}:=H^1_{\et}(X_1(Np^r)_{/\ol{\Q}},\Z_p)^{\ord}$ and set $\wt{H}_{r}=H^1_{\et}(Y_{1}(Np^r)_{/\ol{\Q}},\Z_p)$. One can identify $\wt{H}_r$ with a certain relative homology group by Poincar\'{e} duality and the comparison between Betti (co)homology groups and \'{e}tale (co)homology groups. Namely, one has 
\begin{equation}\label{eq:coho_relative_homo}
\wt{H}_r(1) \isom H_1(Y_1(Np^r)(\C),C_1(Np^r),\Z_p),
\end{equation}
where $C_1(Np^r)$ is the set of cusps for $\Gamma_1(Np^r)$ (see \cite[Sections~3.4 and 3.5]{Sha}). 
From the discussion in Section~3.1 of \textit{loc.~cit.}, the group $H_1(Y_1(Np^r)(\C),C_1(Np^r),\Z_p)^+$ is generated by the adjusted Manin symbols $[u,v]_r^+$ for all $u,v\in \Z/Np^r\Z$ with $(u,v)=1$  satisfying the properties (3.3)-(3.7) in \textit{loc.~cit.}. We will identify these symbols with elements of $\wt{H}_r(1)^+$ via \eqref{eq:coho_relative_homo} without further notice. Let $\wt{H}_{r,0}(1)^+$ be the subgroup of $\wt{H}_r(1)^+$ generated by $[u,v]^+_r$ for all $u,v\in \Z/Np^r\Z-\{0\}$ with $(u,v)=1$. It was proved in Proposition~5.7 of \textit{loc.~cit.}~(or see \cite[Section~5.2.1]{FK}) that there exists a homomorphism
$$
\wt{\varpi}_{r}:\wt{H}_{r,0}(1)^{+}\rightarrow H^2(\Z[\zeta_{Np^r},\tfrac{1}{Np}],\Z_p(2))^+,
$$
sending $[u,v]_r^+$ to the cup product $(1-\zeta_{Np^r}^u,1-\zeta_{Np^r}^v)_r^+$ and satisfying $\wt{\varpi}_r\circ \<j\>_r=\sigma_j\circ \wt{\varpi}_r$ for all $j\in (\Z/Np^r\Z)^{\times}$. Here $\<j\>$ is the diamond operator and $\sigma_j\in \Gal(\Q(\zeta_{Np^r})/\Q)$ satisfying $\sigma_j(\zeta_{Np^r})=\zeta_{Np^r}^j$. Since $\wt{H}_r(1)^+\isom \wt{H}_r^-(1)$ and since $H^1_{\et}(X_1(Np^r)_{/\ol{\Q}},\Z_p)^-(1)$ is contained in $\wt{H}_{r,0}^-(1)$, the homomorphism $\wt{\varpi}_r$ induces via restriction a homomorphism
$$
\varpi_{r}:H^1_{\et}(X_1(Np^r)_{/\ol{\Q}},\Z_p)^-(1)\rightarrow H^2(\Z[\zeta_{Np^r},\tfrac{1}{Np}],\Z_p(2))^+
$$
whose image is contained in $H^2(\Z[\zeta_{Np^r},\tfrac{1}{p}],\Z_p(2))^+$ by \cite[Theorem~5.3.5]{FK}.

Let $\cI^*_{r}$ be the ideal of the Hecke algebra $\mathcal{H}_r^*$ (acting on $\wt{H}_{r,0}$) generated by $T^*_{\ell}-1-\ell\langle \ell \rangle^{-1}$ for all $\ell\nmid Np^r$ and $U^*_q-1$ for all $q| Np$, and denote by $I^*_{r}$ the image of $\cI^*_{r}$ in the cuspidal Hecke algebra $\mfr{h}^*_r$ acting on $H^1_{\et}(X_1(Np^r)_{/\ol{\Q}},\Z_p)$ under the natural map $\mathcal{H}_r^* \surj \mfr{h}^*_r$. It was conjectured by Sharifi (Conjecture~5.8 in \textit{loc.~cit.}) and proved by Fukaya\textendash Kato \cite[Theorem~5.2.3]{FK} that the map $\wt{\varpi}_{r}$ satisfies
\begin{equation}\label{eq:eisenstein_quotient}
\wt{\varpi}_{r}(\eta x)=0
\end{equation}
for all $\eta\in \cI^*_{r}$ and $x\in \wt{H}_{r,0}^-(1)$.
Moreover, they also proved that the following diagram commutes
\[
\begin{tikzcd}
H^1_{\et}(X_1(Np^{r+1})_{/\ol{\Q}},\Z_p)^-(1)\ar[d]\ar[r,"\varpi_{r+1}"]& H^2(\Z[\zeta_{Np^{r+1}},\tfrac{1}{p}],\Z_p(2))^+\ar[d]\\
H^1_{\et}(X_1(Np^r)_{/\ol{\Q}},\Z_p)^-(1)\ar[r,"\varpi_{r}"]& H^2(\Z[\zeta_{Np^r},\tfrac{1}{p}],\Z_p(2))^+,
\end{tikzcd}
\]
where the left and right vertical maps are the trace map and the norm map, respectively. This diagram induces a map by taking projective limit
\[
\varpi: \varprojlim_r H^1_{\et}(X_1(Np^r)_{/\ol{\Q}},\Z_p)^-(1) \rightarrow \varprojlim_r H^2(\Z[\zeta_{Np^r},\tfrac{1}{p}],\Z_p(2))^+=:S.
\]
It follows from \eqref{eq:eisenstein_quotient} that this map factors through the quotient by $I^*:=\varprojlim_r I^*_r$. Since $H^1_{\et}(X_1(Np^r)_{/\ol{\Q}},\Z_p)/I^*_rH^1_{\et}(X_1(Np^r)_{/\ol{\Q}},\Z_p)$ is isomorphic to $H_r/I^*_rH_r$ for all $r\in \Z_{\geq 1}$, we obtain a homomorphism 
$$
\varpi: H^-(1)/I^*H^-(1)\to S.
$$
Recall that in the Introduction we set $\Lambda:=\Z_p[[\varprojlim_r(\Z/Np^r\Z)^{\times}]]\isom \Z_p[ (\Z/Np\Z)^{\times}]\lsem (1+p\Z_p)\rsem$. From the above discussion, we know that both $H^-(1)$ and $S$ are $\Lambda$-modules and the map $\varpi$ satisfies $\varpi \circ \<j\>=\sigma_j\circ \varpi$  for all $j\in (\Z/Np\Z)^{\times}\times (1+p\Z_p)$. Then, for a primitive even Dirichlet character $\theta$ of conductor $Np$, the $\Lambda_{\theta}$-module homomorphism $\varpi_{\theta}$ in Conjecture~\ref{conj:Mc_Sha} is obtained by taking $\theta$ components on the both sides of $\varpi$. As in the Introduction, we will identify $\Lambda_{\theta}$ with $\Z_p[\theta]\lsem T\rsem$ without further notice.

\section{Galois cohomology and cohomology of modular curves}\label{sec:03}
In the rest of this article, we assume that the Dirichlet character $\theta$ is exceptional, i.e., the character $\chi:=\theta\omega^{-1}$ is trivial on $(\Z/p\Z)^{\times}$ and $\chi|_{(\Z/N\Z)^{\times}}(p)=1$. As in the Introduction, we will view $\chi$ as an odd primitive character of conductor $N$ with $\chi(p)=1$ without further notice. The aim of Section~\ref{sec:31} is to show that $S_{\theta,\mfr{p}}$ is a $1$-dimensional $k_{\theta,\mfr{p}}$-vector space. In Section~\ref{sec:32}, we first construct elements in $H^-_{\theta}(1)$ and then show that the $k_{\theta,\mfr{p}}$-vector space $H^-_{\theta,\mfr{p}}(1)/I^*_{\theta,\mfr{p}}H^-_{\theta,\mfr{p}}(1)$ is $1$-dimensional. 
\subsection{Localization of $S_{\theta}$ at $\mfr{p}$}\label{sec:31}
Let $\Sigma$ be the set of the finite places of $\Q(\zeta_{Np^{\infty}})$ above $p$, and let  $X_{\Q(\zeta_{Np^{\infty}})}$ (resp.~$X_{\Q(\zeta_{Np^{\infty}}),\Sigma}$) be the Galois group of the maximal abelian unramified pro-$p$ extension of $\Q(\zeta_{Np^{\infty}})$ (resp.~in which all primes above those in $\Sigma$ split completely).  

By \cite[Lemma 2.1]{Sha}, We have the following exact sequence
\begin{equation}\label{eq:long_exact_seq_S}
\begin{split}
0\rightarrow X_{\Q(\zeta_{Np^{\infty}}),\Sigma}\rightarrow& \varprojlim_r H^2(\Z[\zeta_{Np^r},\tfrac{1}{p}],\Z_p(1))\\
\rightarrow& \bigoplus_{v\in \Sigma} \varprojlim_r H^2(\Q(\zeta_{Np^r})_{v},\Z_p(1))\xrightarrow{(*)} \Z_p\rightarrow 0.
\end{split}
\end{equation}
Note that one has 
$$
\varprojlim_r H^2(\Z[\zeta_{Np^r},\tfrac{1}{p}],\Z_p(2))_{\theta}=\varprojlim_r H^2(\Z[\zeta_{Np^r},\tfrac{1}{p}],\Z_p(1))_{\chi}(1)
$$
and by local class field theory, one has
\begin{equation}\label{eq:brauer_group}
H^2(\Q(\zeta_{Np^r})_{v},\Z_p(1)):=\varprojlim_m H^2(\Q(\zeta_{Np^r})_{v},(\Z/p^m\Z)(1))
\cong \Z_p
\end{equation}
for all $v\in \Sigma$. Under the identification \eqref{eq:brauer_group}, the map $(*)$ in \eqref{eq:long_exact_seq_S} is the sum map whose kernel is denoted by $(\oplus^0_{v\in \Sigma}\Z_p)$. After taking $\chi$ components and the Tate twist on the exact sequence \eqref{eq:long_exact_seq_S}, we obtain the following short exact sequence
\begin{equation}\label{eq:long_ext_seq_S_theta}
0 \rightarrow X_{\Q(\zeta_{Np^{\infty}}),\Sigma,\chi}(1)\rightarrow S_{\theta}\rightarrow (\oplus^0_{v\in \Sigma}\Z_p)_{\chi}(1)\rightarrow 0.
\end{equation}
The following lemma shows that the first term in \eqref{eq:long_ext_seq_S_theta} is trivial after localizing at $\mfr{p}$.

\begin{lemma}\label{vanishing_X_Sigma}
Let the notation be as above. Then, the characteristic ideal of the $\Lambda_{\theta}$-module $X_{\Q(\zeta_{Np^{\infty}}),\chi}/X_{\Q(\zeta_{Np^{\infty}}),\Sigma,\chi}$ is $(T)$. Moreover, one has $X_{\Q(\zeta_{Np^{\infty}}),\Sigma,\chi}(1)_{\mfr{p}}=0$.
\end{lemma}

\begin{proof}
For $n\in \Z_{\geq 1}$, let $A_n$ be the $p$-part of the class group of the \textit{n}th-layer of the cyclotomic $\Z_p$-extension over $\Q$, and let $D_n$ be the subgroup of $A_n$ generated by all primes above $p$. From the discussion in the last paragraph of p.~99 in \cite[Section~4]{FG}, one can see that $X_{\Q(\zeta_{Np^{\infty}}),\chi}/X_{\Q(\zeta_{Np^{\infty}}),\Sigma,\chi}$ is a direct summand of the Pontryagin dual of the direct limit of $D_n^-$ whose characteristic ideal is a power of $(T)$. Since $\theta$ is exceptional, it was also shown in \textit{loc.~cit.}~that $(T)$ divides the characteristic ideal of $X_{\Q(\zeta_{Np^{\infty}}),\chi}$ exactly once. Hence, the characteristic ideal of $X_{\Q(\zeta_{Np^{\infty}}),\chi}/X_{\Q(\zeta_{Np^{\infty}}),\Sigma,\chi}$ is $(T)$. Thus, $X_{\Q(\zeta_{Np^{\infty}}),\Sigma,\chi}\otimes_{\Lambda_{\theta}} \Lambda_{\theta,(T)}=0$ which would imply the second assertion by taking the Tate twist.
\end{proof}

\begin{prop}\label{Brauer_structure}
We have isomorphisms of $\Lambda_{\theta,\p}$-modules
\[
S_{\theta,\mfr{p}}
\isom \big((\oplus^0_{v\in\Sigma}\Z_p)_{\chi}(1)\big)_{\p}
\isom \Lambda_{\theta,\p}/\p.
\]
\end{prop}

\begin{proof}
The first isomorphism follows from Lemma~\ref{vanishing_X_Sigma} and \eqref{eq:long_ext_seq_S_theta}. We now prove the second isomorphism. Note that as $\Z_p[\chi]$-module, $(\oplus^0_{v\in \Sigma}\Z_p)_{\chi}$ is isomorphic to $\Z_p[\chi]$, and for any $\sigma\in \Gal(\Q(\zeta_p)/\Q)$, it acts on $(\oplus^0_{v\in \Sigma}\Z_p)_{\chi}(1)$ via $\omega(\sigma)$. Moreover, the topological generator $\gamma$ of $ \Gal((\Q(\zeta_{Np^{\infty}})/\Q(\zeta_{Np})))$  acts on $(\oplus^0_{v\in S}\Z_p)_{\chi}(1)$ via $\kappa(\gamma)$. Therefore, one has
\[
(\oplus^0_{v\in \Sigma}\Z_p)_{\chi}(1)\cong \Lambda_{\theta}/\p,
\]
and hence, the assertion follows by taking the localization at $\mfr{p}$.
\end{proof}

\subsection{Cohomology of modular curves}\label{sec:32}
We first recall some results of Lafferty \cite{Laf} and Ohta \cite{Oht2}, which will be used in the proof of Proposition~\ref{DM}. To do this, we start by recalling the definition of ordinary $\Lambda$-adic modular forms. For this topic, we refer the reader to \cite[Chapter~7]{hida} for more details.

Let $\cO$ be an extension of $\Z_p[\theta]$ containing all \textit{p}-power roots of unity. For $k\in \Z_{\geq 2}$ and for a character $\epsilon$ of $1+p\Z_p$ of order $p^{r-1}$ with $r\in \Z_{>0}$, the specialization map $v_{k,\epsilon}:\cO\lsem T \rsem \to \cO$ is defined by sending $T$ to $\epsilon(\kappa(\gamma))\kappa(\gamma)^{k-2}-1$. An ordinary $\Lambda$-adic modular form (resp.~cusp form) $\mathcal{F}$ of level $N$ and character $\theta$ is a formal power series 
$$
\mathcal{F}=\sum_{n=0}^{\infty} A_n(T;\mathcal{F}) e^{2\pi i nz}
$$
with $A_n(T;\mathcal{F})\in \cO\lsem T \rsem$ for all $n\in \Z_{\geq 0}$ such that $v_{k,\epsilon}(\mathcal{F})$ is an ordinary modular form (resp.~cusp form) of weight $k$, level $\Gamma_1(Np^r)$, and character $\theta\omega^{2-k}\epsilon$ for all $k\in \Z_{\geq 2}$ and for all but finitely many characters $\epsilon$ of $1+p\Z_p$ of order $p^{r-1}$. Recall that a normalized eigenform of weight $k$ and level $\Gamma_1(Np^r)$ is called ordinary if its \textit{p}-th Fourier coefficient is a $p$-adic unit. Following \cite[(1.4.2)]{Oht2}, an example of ordinary $\Lambda$-adic modular forms is given by the ordinary $\Lambda$-adic Eisenstein series
$$
\mathcal{E}(\theta,\mathbbm{1})=2^{-1} G_{\theta}(T)+\sum_{n=1}^{\infty} \left(\sum_{d|n, (d,p)=1} \theta(d) (1+T)^{s(d)} d\right) e^{2\pi i nz},
$$
where $s(d):=\tfrac{\log_p(d)}{\log_p(\kappa(\gamma))}$ and $G_{\theta}(T)\in \Lambda_{\theta}$ is the power series expression of the Kubota\textendash Leopoldt $p$-adic $L$-function $L_p(s,\theta\omega^2)$ satisfying
\begin{equation}\label{eq:const_Eisenstein_family_at_infty}
G_{\theta}(\kappa(\gamma)^s-1)=L_p(-s-1,\theta\omega^2).
\end{equation}
Recall that for an even Dirichlet character $\psi$ the Kubota\textendash Leopoldt $p$-adic $L$-function $L_p(s,\psi)$ satisfies the interpolation property 
$$
L_p(1-k,\psi)=(1-\psi\omega^{-k}(p)p^{k-1})L(1-k,\psi\omega^{-k})
$$
for all $k\in \Z_{>0}$.
Set $\xi_{\theta}(T):=G_{\theta^{-1}}((1+T)^{-1}-1)$.
Then, it follows from the assumption $\chi(p)=1$ that $T=\kappa(\gamma)-1$ is a zero of $\xi_{\theta}(T)$, called the trivial zero. Let $M^{\ord}(N,\theta;\cO\lsem T \rsem)$ be the space of ordinary $\Lambda$-adic modular forms of level $N$ and character $\theta$, and let $S^{\ord}(N,\theta;\cO\lsem T \rsem)$ be the subspace of $M^{\ord}(N,\theta;\cO\lsem T \rsem)$ consisting of ordinary $\Lambda$-adic cusp forms. Recall that we denote by $C_1(Np^r)$ the set of cusps for $\Gamma_1(Np^r)$. It was proved by Ohta \cite[(2.4.4)]{Oht2} that one has the following short exact sequence of free $\cO\lsem T \rsem$-modules
\begin{equation}\label{eq:surj_res}
0\to S^{\ord}(N,\theta;\cO\lsem T \rsem) \to M^{\ord}(N,\theta;\cO\lsem T \rsem) \xrightarrow{\Res} e\cdot \cO\lsem C_{\infty}\rsem\to 0, 
\end{equation}
where $e=\lim_{n\to \infty} U_p^{n!}$ is the ordinary projector, $\cO\lsem C_{\infty}\rsem:=\varprojlim_r \cO[C_1(Np^r)]$ is a free $\cO\lsem T \rsem$-module of finite rank for which the projective limit is with respect to the natural projection $C_1(Np^r)\surj C_1(Np^s)$ for $r\geq s$, and $\Res$ is the residue map ((2.4.9) in \textit{loc.~cit.}) for $\Lambda$-adic modular forms defined by 
$$
\Res(\mathcal{F}):=\varprojlim_{r} \left( \frac{1}{p^r-1}\sum_{c\in C_1(Np^r)}\left(\sum_{\epsilon} \mathrm{res}_{W_{Np^r}\cdot c} (v_{2,\epsilon}(\mathcal{F})|U_p^{-r})
\right)[c]\right).
$$
Here $W_{Np^r}:=\left(\begin{smallmatrix} 0 & -1 \\ Np^r & 0 \end{smallmatrix}\right)$ is the 
Atkin\textendash Lehner involution, the above second summand runs through all characters $\epsilon$ of $1+p\Z_p$ of order $p^{r-1}$, and $\mathrm{res}_c f$ is the residue of $f\tfrac{dq}{q}$ at the cusp $c$.

We denote by $M^{\ord}(N,\theta;\Lambda_{\theta})$ (resp.~$S^{\ord}(N,\theta;\Lambda_{\theta})$) the subspace of $M^{\ord}(N,\theta;\cO\lsem T \rsem)$ (resp.~$S^{\ord}(N,\theta;\cO\lsem T \rsem)$) consisting of ordinary $\Lambda$-adic modular forms (resp.~cusp forms) $\mathcal{F}$ with $A_n(T;\mathcal{F})$ in $\Lambda_{\theta}$ for all $n\in \Z_{\geq 0}$.
It was proved by Ohta \cite[Lemma~2.1.1]{Oht3} that $\cO\lsem T \rsem$ is a faithfully flat $\Lambda_{\theta}$-module. Therefore, one obtains from \eqref{eq:surj_res} a short exact sequence of free $\Lambda_{\theta}$-modules
$$
0\to S^{\ord}(N,\theta;\Lambda_{\theta}) \to M^{\ord}(N,\theta;\Lambda_{\theta}) \xrightarrow{\Res} e\cdot \Z_p[\theta]\lsem C_{\infty}\rsem\to 0.
$$

Let $e_{(\theta,\mathbbm{1})} \in e\cdot \Z_p[\theta]\lsem C_{\infty}\rsem$ be defined in \cite[Proposition~3.2.1]{Laf} such that $\Lambda_{\theta}\cdot e_{(\theta,\mathbbm{1})}$ is a direct summand of $e\cdot \Z_p[\theta]\lsem C_{\infty}\rsem$ and  $\Res(\mathcal{E}(\theta,\mathbbm{1}))=G_{\theta}(T) \cdot e_{(\theta,\mathbbm{1})}$. Let $M_{\Lambda_{\theta}}$ be the preimage of $\Lambda_{\theta}\cdot e_{(\theta,\mathbbm{1})}$ under the map $\Res$, and set $S_{\Lambda_{\theta}}:=S^{\ord}(N,\theta;\Lambda_{\theta})$. Then we obtain a short exact sequence of free $\Lambda_{\theta}$-modules
\begin{equation}\label{eq:modular_forms_seq}
    0\to S_{\Lambda_{\theta}}\to M_{\Lambda_{\theta}} \xrightarrow{\Res} \Lambda_{\theta}\cdot e_{(\theta,\mathbbm{1})} \to 0.
\end{equation}
Let $\cH_{\theta}=\Z_p[T_n\mid n\geq 1]$ be the Hecke algebra acting on $M_{\Lambda_{\theta}}$, and let $\mfr{h}_{\theta}$ be the cuspidal Hecke algebra action on $S_{\Lambda_{\theta}}$ defined in the same manner. Set $\mathcal{I}_{\theta}:=\Ann_{\cH_{\theta}}(\cE(\theta,\mathbbm{1}))$ and denote by $I_{\theta}$ the image of $\cI_{\theta}$ under the natural homomorphism $\cH_{\theta}\surj \mfr{h}_{\theta}$. Denote by $s':M_{\Lambda_{\theta}}\otimes_{\Lambda_{\theta}} Q(\Lambda_{\theta})\to S_{\Lambda_{\theta}}\otimes_{\Lambda_{\theta}} Q(\Lambda_{\theta})$ the unique Hecke equivariant splitting map and set $M_{\Lambda_{\theta},DM'}:=s'(M_{\Lambda_{\theta}})$. Here $Q(\Lambda_{\theta})$ denotes the quotient field of $\Lambda_{\theta}$. It was proved in Section~3 of \textit{loc.~cit.}~that one has isomorphisms of $\Lambda_{\theta}$-modules 
\begin{equation}\label{eq:Matt_thesis}
M_{\Lambda_{\theta},DM'}/S_{\Lambda_{\theta}}\isom  \Lambda_{\theta}/(G_{\theta}(T))\isom 
\mfr{h}_{\theta}/I_{\theta}.
\end{equation}

Recall that in Section~\ref{sec:02}, we set $H_{\theta}:= \varprojlim_r H^1_{\et}(X_1(Np^{r})_{/\ol{\Q}},\Z_p)^{\ord}_{\theta}$. By \cite[(4.3.12)]{Oht1}, one has the following short exact sequence of $\Lambda_{\theta}[G_{\Q}]$-modules
\begin{equation}\label{eq:Lambda_Eich_Shi_isom}
0 \to  H_{\theta} \to  \varprojlim_r H^1_{\et}(Y_1(Np^{r})_{/\ol{\Q}},\Z_p)^{\ord}_{\theta} \to e\cdot \Z_p[\theta]\lsem C_{\infty}\rsem(-1) 
\to 0.
\end{equation}
Let $\wt{H}_{\theta}\subset \varprojlim_r H^1_{\et}(Y_1(Np^{r})_{/\ol{\Q}},\Z_p)^{\ord}_{\theta}$ be the preimage of $\Lambda\cdot e_{(\theta,\mathbbm{1})}$. Then, one obtains from \eqref{eq:Lambda_Eich_Shi_isom} a short exact sequence of $\Lambda_{\theta}[G_{\Q}]$-modules
\begin{equation}\label{eq:coho_exact_seq}
0 \to H_{\theta} \to  \wt{H}_{\theta} \to \Lambda_{\theta}\cdot e_{(\theta,\mathbbm{1})}(-1)
\to 0.
\end{equation}
We denote by $\cH^*_{\theta}$ (resp.~$\mfr{h}^*_{\theta}$) the Hecke algebra acting on $\wt{H}_{\theta}$ (resp.~on $H_{\theta}$). Also, we denote by $\cI^*_{\theta}$ and $I^*_{\theta}$ the corresponding Eisenstein ideals in $\cH^*_{\theta}$ and $\mfr{h}^*_{\theta}$, respectively. Note that the last map in \eqref{eq:coho_exact_seq} commutes the action of $\cH^*_{\theta}$ on $\wt{H}_{\theta}$ and the action of $\cH_{\theta}$ on $\Lambda_{\theta}\cdot e_{(\theta,\mathbbm{1})}(-1)$. Moreover, by \eqref{eq:Matt_thesis}, we obtain an isomorphism of $\Lambda_{\theta}$-modules
\begin{equation}\label{eq:cong_mod_dual_hecke_algebra}
\mfr{h}^*_{\theta}/I^*_{\theta}\isom \Lambda_{\theta}/(\xi_{\theta})
\end{equation}
induced by the canonical isomorphism $\mfr{h}_{\theta}\isom \mfr{h}^*_{\theta}$ sending $U_q$ to $U^*_q$ for all $q|Np$, $T_{\ell}$ to $T^*_{\ell}$ and $\<\ell\>$ to $\<\ell\>^{-1}$ for all $\ell\nmid Np$.

The Drinfield\textendash Manin modification $\wt{H}_{\theta,DM}$ of $\wt{H}_{\theta}$ is defined as $\wt{H}_{\theta}\otimes_{\cH^*_{\theta}} \mfr{h}^*_{\theta}$. For each $r\in \Z_{>0}$, we denote by $s^r:H^1(Y_1(Np^r),\Q_p)\to H^1(X_1(Np^r),\Q_p)$ the Drinfield\textendash Manin splitting which induces a splitting map $s:\varprojlim_r H^1(Y_1(Np^r),\Q_p)\to \varprojlim_r H^1(X_1(Np^r),\Q_p)$.
Set $\wt{H}_{\theta,DM'}:= s(\wt{H}_{\theta})$. 

\begin{prop}\label{DM}
Let the notation be as above. We have $\Lambda_{\theta}$-module isomorphisms
\begin{equation}\label{eq:Drinfield_Manin1}
\widetilde{H}_{\theta,DM}/H_{\theta}\isom\Lambda_{\theta}/(\xi_{\theta})
\end{equation}
and 
\begin{equation}\label{eq:Drinfield_Manin2}
\widetilde{H}_{\theta,DM'}/H_{\theta}\isom\Lambda_{\theta}/(\xi_{\theta}).
\end{equation}
In particular, we have an isomorphism of $\Lambda_{\theta}$-modules $\widetilde{H}_{\theta,DM'}\isom \widetilde{H}_{\theta,DM}$, and the $\Lambda_{\theta}/(\xi_{\theta})$-module $\widetilde{H}_{\theta,DM}/H_{\theta}$ is generated by $\{0,\infty\}_{\theta,DM}$, where $\{0,\infty\}$ is the modular symbol attached to the cusps $0$ and $\infty$.
\end{prop}

\begin{proof}
We first prove the isomorphism (\ref{eq:Drinfield_Manin1}). Note that one has natural isomorphisms of $\Lambda_{\theta}$-modules $\cH^*_{\theta}/\cI^*_{\theta}\isom \cH_{\theta}/\cI_{\theta}\isom  \Lambda_{\theta}\cdot e_{\theta,\mathbbm{1}}$. Thus, the sequence (\ref{eq:coho_exact_seq}) yields the following short exact sequence of free $\Lambda_{\theta}$-modules
$$
0\map H_{\theta}\map \widetilde{H}_{\theta}\map \cH^*_{\theta}/\cI^*_{\theta} \map 0.
$$
By tensoring $\mfr{h}^*_{\theta}$ over $\cH^*_{\theta}$ on the above sequence, one obtains
$$
0\map H_{\theta}\map \widetilde{H}_{\theta,DM}\map \mfr{h}^*_{\theta}/I^*_{\theta} \map 0,
$$
which yields (\ref{eq:Drinfield_Manin1}) by \eqref{eq:cong_mod_dual_hecke_algebra}. 

Next, we prove the isomorphism \eqref{eq:Drinfield_Manin2}. It follows from \cite[Lemma~1.1.4]{Oht1} that the congruence module  $\widetilde{H}_{\theta,DM'}/H_{\theta}$ is isomorphic to $\Lambda_{\theta}/(f)$ for some $f\in \Lambda_{\theta}$. The Drinfield\textendash Manin splitting induces a surjective homomorphism of $\Lambda_{\theta}$-modules $\widetilde{H}_{\theta,DM}\surj \widetilde{H}_{\theta,DM'}$ (see the proof of \cite[Lemma~4.1]{Sha}), which would yield another surjective homomorphism of $\Lambda_{\theta}$-modules
$$
\widetilde{H}_{\theta,DM}/H_{\theta}\surj \widetilde{H}_{\theta,DM'}/H_{\theta}.
$$
Thus, it follows from \eqref{eq:Drinfield_Manin1} that $f|\xi_{\theta}$. We claim that one also has $\xi_{\theta}|f$, which implies \eqref{eq:Drinfield_Manin2} by the above discussion. It was proved by Ohta that one has surjective homomorphisms $\wt{H}_{\theta}\surj M_{\Lambda}$ and $H_{\theta}\surj S_{\Lambda}$ on which the action of $T^*_{\ell}$ and $U^*_q$ on the left commute with $T_{\ell}$ and $U_q$ on the right for all $\ell\nmid Np$ and for all $q|Np$. They induce surjective homomorphisms $\wt{H}_{\theta,DM'}\surj M_{\Lambda,DM'}$ and 
$\wt{H}_{\theta,DM'}/H_{\theta}\surj M_{\Lambda,DM'}/S_{\Lambda}$.
Thus, the claim follows from \eqref{eq:Matt_thesis}. The isomorphism $\widetilde{H}_{\theta,DM'}\isom \widetilde{H}_{\theta,DM}$ follows from \eqref{eq:Drinfield_Manin1}, \eqref{eq:Drinfield_Manin2}, and the Snake lemma.

Finally, we note that $\{0,\infty\}_{\theta}$ is part of a basis of $\varprojlim_r H^1_{\et}(Y_1(Np^{r})_{/\ol{\Q}},\Z_p)^{\ord}_{\theta}$ whose image under the boundary map is in $\Lambda_{\theta}\cdot e_{(\theta,\mathbf{1})}$. Therefore, $\{0,\infty\}_{\theta,DM}$ is part of a basis of $\wt{H}_{\theta,DM}$, and hence, $\wt{H}_{\theta,DM}/H_{\theta}$ is generated by $\{0,\infty\}_{\theta,DM}$.
\end{proof}

The following corollary will be used in Section~\ref{sec:04}. 

\begin{cor}\label{image_of_varpi}
The elements $\xi_{\theta}\{0,\infty\}_{\theta,DM}$ and $(1-U^*_q)\cdot \{0,\infty\}_{\theta,DM}$ for all $q|Np$ are in $H_{\theta}^-$. Moreover, we have
\begin{equation}\label{eq:image_of_varpi}
\varpi_{\theta}((1-U^*_q)\cdot \{0,\infty\}_{\theta,DM})=(q,1-\zeta_{Np^r})_{r\geq 1,\theta}\in S_{\theta}
\end{equation}
for all $q|Np$.
\end{cor}

\begin{proof}
Note that $1-U^*_q$ is in $I^*_{\theta}$ for all $q|Np$. By Proposition~\ref{DM}, the proof of the first assertion is essentially the same as the proof of \cite[Lemma~4.8]{Sha}.
For $q=p$, \eqref{eq:image_of_varpi} was proved in \cite[Section~10.3]{FK} so it remains to deal with the case $q|N$. Given any $r\in \Z_{>0}$, by the definition of the Hecke actions on the set of cusps (see \cite[Section~2.1]{Oht3} for example), we have $(1-U_q^*)\cdot \{0,\infty\}_r=\sum_{i=1}^{q-1} \{\tfrac{i}{q},\infty\}_r$. From the discussion in \cite[Section~3.1]{Sha} and using (3.1)-(3.3) in \textit{loc.~cit.}, one can write the above modular symbols as Manin symbols. Namely, one has $\sum_{i=1}^{q-1}\{\tfrac{i}{q},\infty\}_r=\sum_{i=1}^{q-1} [N'p^r,1]_r$, where $N'=N/q$. Furthermore, by the definition of $\varpi$ (see Section~2.1), one has
$$
\varpi\left(\sum_{i=1}^{q-1} [N'p^r,1]_{r\geq 1}\right)
=\sum_{i=1}^{q-1} (1-\zeta_q^i,1-\zeta_{Np^r})_{r\geq 1}
=\left( \prod_{i=1}^{q-1} (1-\zeta_q^i),\zeta_{Np^r}\right)_{r\geq 1}
=(q,1-\zeta_{Np^r})_{r\geq 1}.
$$
Then, by taking $\theta$-components and taking Drinfield\textendash Manin modification, we obtain \eqref{eq:image_of_varpi} for all $q|N$.
\end{proof}

To close this section, we next show that $H^-_{\theta,\mfr{p}}/I^*_{\theta,\mfr{p}}H^-_{\theta,\mfr{p}}$ is a $1$-dimensional vector space over $k_{\theta,\mfr{p}}$.
It was proved by Ferrero\textendash Greenberg \cite{FG} that the trivial zero of Kubota\textendash Leopoldt $p$-adic $L$-functions is a simple zero if it exists. Therefore, $I^*_{\theta,\mfr{p}}$ is the maximal ideal of $\mfr{h}^*_{\theta,\mfr{p}}$ by \eqref{eq:cong_mod_dual_hecke_algebra}. We notice that $\mfr{h}^*_{\theta,\mfr{p}}$ and $\cT^{\cusp}$ in \cite{BDP} are isomorphic as $\Lambda_{\theta,\mfr{p}}$-modules but not identically the same since the former acts on the cohomology of modular curves and the latter acts on the space of ordinary cuspidal families. Also, the convention of the specialization maps are different, which can be seen by comparing \eqref{eq:const_Eisenstein_family_at_infty} with (39) in \textit{loc.~cit.} By Theorem~A(i) (or Theorem~3.5(ii)) in \textit{loc.~cit.}, one has $\Lambda_{\theta,\mfr{p}}\isom \mfr{h}^*_{\theta,\mfr{p}}$ and hence, one has $I^*_{\theta,\mfr{p}}=\mfr{p}$. One can deduce from this that $H_{\theta,\mfr{p}}$ is a free $\mfr{h}^*_{\theta,\mfr{p}}$-module of rank $2$, since it is known that $H_{\theta}$ is a torsion free $\mfr{h}^*_{\theta}$-module and $H_{\theta} \otimes_{\mfr{h}^*_{\theta}} Q(\mfr{h}^*_{\theta})$ is a $2$-dimensional vector space over $Q(\mfr{h}^*_{\theta})$, where $Q(\mfr{h}^*_{\theta})$ is the quotient field of $\mfr{h}^*_{\theta}$. Therefore, both $H_{\theta,\mfr{p}}^+$ and $H_{\theta,\mfr{p}}^-$ are free $\mfr{h}^*_{\theta,\mfr{p}}$-modules of rank $1$. The following proposition follows from the above discussion immediately.

\begin{prop}\label{source_of_varpi}
One has $\dim_{k_{\theta,\mfr{p}}} H^-_{\theta,\mfr{p}}/I^*_{\theta,\mfr{p}}H^-_{\theta,\mfr{p}}=1=\dim_{k_{\theta,\mfr{p}}} H^+_{\theta,\mfr{p}}/I^*_{\theta,\mfr{p}}H^+_{\theta,\mfr{p}}$.
\end{prop}

\section{Surjectivity of $\varpi_{\theta,\mfr{p}}$}\label{sec:04}
The goal of this section is to show that the image of $\varpi_{\theta,\mfr{p}}$ is non-trivial, which completes the proof of Theorem~\ref{11}. 

We first recall some useful formulas for later use. For positive integers $a$ and $M$ with $a<M$, the partial Riemann zeta function  $\zeta_{a\; (M)}(s)$ is defined by
$$
\zeta_{a\; (M)}(s):=\sum_{n\equiv a\bmod M} n^{-s}.
$$
It converges absolutely  when $\mathrm{Re}(s)>1$ and admits an holomorphic continuation on $\C-\{1\}$ which has a simple pole at $s=1$ with residue $1/M$ \cite[Section~1.3.1]{Ka}. 

\begin{lemma}\label{useful_formula_local_pairing}
Let the notation be as above. Then the following assertions hold.
\begin{enumerate}

\item If $a=a'p^r$ for some positive integer $a' < N$, then $\zeta_{a\; (Np^r)}(s)=p^{-rs} \zeta_{a'\; (N)}(s)$.
    

\item For any $g_N\in (\Z/N\Z)^{\times}$ and $g_{p^r}\in (\Z/p^r\Z)^{\times}$, we have
$$
\frac{\zeta_N^{g_N}\zeta_{p^r}^{g_{p^r}}}{\zeta_N^{g_N}\zeta_{p^r}^{g_{p^r}}-1}
=-\sum_{a\in \Z/Np^r\Z} \zeta_{a\; (Np^r)}(0) (\zeta_N^{g_N}\zeta_{p^r}^{g_{p^r}})^a.
$$

\item 
For $r\in \Z_{\geq 1}$, one has 
$$
\sum_{i\in (\Z/p^r\Z)^{\times}} \zeta_{p^r}^i=
\begin{cases}
-1 & \mbox{ if } r=1,\\
0  & \mbox{ otherwise}.
\end{cases}
$$
\end{enumerate}
\end{lemma}

\begin{proof}
Note that for each positive integer $n$ with $n\equiv a \bmod Np^r$, one can write $n$ as $a+iNp^r$ for some $i\in \Z_{\geq 0}$. Suppose $a=a'p^r$ for some positive integer $a'<N$. Then, one has 
\[
\begin{split}
\zeta_{a\; (Np^r)}(s) 
= \sum_{n\equiv a'p^r \bmod Np^r} n^{-s}
= \sum_{i=0}^{\infty} (a'p^r+iNp^r)^{-s}
= p^{-rs} \sum_{i=0}^{\infty} (a'+iN)^{-s}
=p^{-rs} \zeta_{a'\; (N)}(s).
\end{split}
\]
This proves the assertion (1). The assertion (2) follows from \cite[Lemma~1.3.15(1)]{Ka} by taking $r=1$ and $t=\zeta_N^{g_N}\zeta_{p^r}^{g_{p^r}}$.
To prove the assertion (3), it is known that $\sum_{i=0}^{p^r-1}\zeta_{p^r}^i=0$. When $r=1$, this implies that $\sum_{i=1}^{p-1}\zeta_{p}^i=-1$. If $r\in \Z_{>1}$, then one has
$$
\sum_{p|i,1\leq i\leq p^r-1} \zeta_{p^r}^i
=\sum_{i=1}^{p^{r-1}-1}\zeta_{p^{r-1}}
=-1.
$$
which implies $\sum_{i\in (\Z/p^r\Z)^{\times}} \zeta_{p^r}^i=0$. 
\end{proof}

For each $r\in \Z_{\geq 1}$ and each $q|N$, the cup product $(q,1-\zeta_N^{p^{-r}}\zeta_{p^r})_{r,\theta,\mfr{p}}$ can be identified via the first isomorphism in Proposition~\ref{Brauer_structure} with the local Hilbert symbol which will be computed by the local explicit reciprocity law in the following theorem.

\begin{thm}\label{local_pairing}
For each positive integer $r$ and each prime $\ell |N$, we have
$$
(\ell ,1-\zeta_N^{p^{-r}}\zeta_{p^r})_{r,\theta,\mfr{p}}=\frac{(p-1)\log_p (\ell) }{p \phi(N)} \tau(\chi^{-1}) L(0,\chi)\in (\cO/p^r\cO)(1).
$$
\end{thm}

\begin{proof}
Note that the action of $(\Z/p\Z)^{\times}$ on the group of $p^r$th roots of unity  is given by $a\cdot \zeta_{p^r}=\zeta_{p^r}^{\omega(a)}$ for all $a\in (\Z/p\Z)^{\times}$. Then, one has
$$
(\ell ,1-\zeta_N^{p^{-r}}\zeta_{p^r})_{r,\theta,\mfr{p}}
= \phi(Np)^{-1}\sum_{g_N,g_p}\chi(g_N^{-1})\omega(g_p^{-1})(\ell ,1-\zeta_N^{p^{-r}g_N}\zeta_{p^r}^{\omega(g_p)})_{r,\mfr{p}}.
$$
Here $g_N$ and $g_p$ run through all elements in $(\Z/N\Z)^{\times}$ and $(\Z/p\Z)^{\times}$, respectively. Set $G_N=p^{-r} g_N$. Using the assumption that $\chi(p)=1$, one can rewrite the above summation as
\begin{equation}\label{eq:local_paring_1}
 \phi(Np)^{-1}\sum_{G_N,g_p}\chi(G_N^{-1})\omega(g_p^{-1})(\ell ,1-\zeta_N^{G_N}\zeta_{p^r}^{\omega(g_p)})_{r,\mfr{p}}.
\end{equation}
Here $G_N$ runs through all elements in $(\Z/N\Z)^{\times}$, since it follows from the assumption $p\nmid \phi(N)$ that $p^r\in (\Z/N\Z)^{\times}$.

Set $\beta_{Np^r}=\zeta_N^{G_N}\zeta_{p^r}^{\omega(g_p)}-1$. The Coleman power series $g_{\beta_{Np^r}}(T)$ associated to $\beta_{Np^r}$ is defined by 
$$
g_{\beta_{Np^r}}(T)=\zeta_N^{G_N}(1+T)^{\omega(g_p)}-1.
$$
For the definition of Coleman power series, we refer the reader to \cite[Section~4]{tsuji} (or see \cite[Section~1]{deS}). Then, one has $g_{\beta_{Np^r}}(\zeta_{p^r}-1)=\beta_{Np^r}$. Set $\delta g_{\beta_{Np^r}}:=(1+T)\tfrac{d g_{\beta_{Np^r}}(T)}{dT}\times g_{\beta_{Np^r}}(T)^{-1}$. Then, one has
$$
\delta g_{\beta_{Np^r}}(\zeta_{p^r}-1)=\omega(g_p)\tfrac{\zeta_N^{G_N}\zeta_{p^r}^{\omega(g_p)}}{\zeta_N^{G_N}\zeta_{p^r}^{\omega(g_p)}-1}.
$$
By the local explicit reciprocity law \cite[Theorem~8.18]{Iwa} or \cite[Theorem~I.4.2]{deS} (taking  $\lambda(X)=1+X$) one can write \eqref{eq:local_paring_1} as
\begin{equation}\label{eq:local_paring_2}
 \frac{1}{p^r \phi(Np)}\sum_{G_N,g_p}\chi(G_N^{-1})\omega(g_p^{-1}) \Tr_{\Q_p(\mu_{p^r})/\Q_p}\left(\log_p (\ell) \cdot \omega(g_p) \frac{\zeta_N^{G_N}\zeta_{p^r}^{\omega(g_p)}}{\zeta_N^{G_N}\zeta_{p^r}^{\omega(g_p)}-1}\right).
\end{equation}
One can simplify \eqref{eq:local_paring_2} by computing the trace as
\begin{equation}\label{eq:local_paring_3}
\frac{\log_p (\ell)}{p^r \phi(Np)}\sum_{G_N,g_p}\chi(G_N^{-1}) \sum_{g_{p^r}\in (\Z/p^r\Z)^{\times}} \frac{\zeta_N^{G_N}\zeta_{p^r}^{\omega(g_p)g_{p^r}}}{\zeta_N^{G_N}\zeta_{p^r}^{\omega(g_p)g_{p^r}}-1}.
\end{equation}
Note that for a fixed $g_p\in (\Z/p\Z)^{\times}$, $\omega(g_p)\cdot (\Z/p^r\Z)^{\times}=(\Z/p^r\Z)^{\times}$. Hence, by setting $G_{p^r}=\omega(g_p)g_{p^r}$, one can simplify \eqref{eq:local_paring_3} as
\begin{equation}\label{eq:local_pairing_4}
\frac{(p-1)\log_p (\ell)}{p^r \phi(Np)}\sum_{G_N}\chi(G_N^{-1}) \sum_{G_{p^r}\in (\Z/p^r\Z)^{\times}} \frac{\zeta_N^{G_N}\zeta_{p^r}^{G_{p^r}}}{\zeta_N^{G_N}\zeta_{p^r}^{G_{p^r}}-1}.
\end{equation}
By Lemma~\ref{useful_formula_local_pairing}(2), one can write \eqref{eq:local_pairing_4} as
\begin{equation}\label{eq:local_pairing_5}
\frac{(p-1)\log_p (\ell)}{p^r \phi(Np)}\sum_{G_N,G_{p^r}}\chi(G_N^{-1})\sum_{a\in \Z/Np^r\Z} \zeta_{a\; (Np^r)} (0) (\zeta_N^{G_N}\zeta_{p^r}^{G_{p^r}})^a.
\end{equation}

One can rewrite \eqref{eq:local_pairing_5} in terms of $L(0,\chi)$ by the following observation. 
\begin{itemize}
\item If $a$ is divisible by $N$, then 
$$
\sum_{G_N,G_{p^r}}\chi(G_N^{-1}) \zeta_{a\; (Np^r)} (0) (\zeta_N^{G_N}\zeta_{p^r}^{G_{p^r}})^a
= \sum_{G_N}\chi(G_N^{-1}) \sum_{G_{p^r}} \zeta_{a\; (Np^r)} (0) \zeta_{p^r}^{a G_{p^r}}
=0,
$$
since $\sum_{G_N}\chi(G_N^{-1})=0$. 
\end{itemize}
Now, we assume that $a$ is not divisible by $N$. By setting $G'_N=G_N\cdot a$, we have
\[
\begin{split}
\sum_{G_N,G_{p^r}}\chi(G_N^{-1}) \zeta_{a\; (Np^r)} (0) (\zeta_N^{G_N}\zeta_{p^r}^{G_{p^r}})^a
=& \chi(a)\zeta_{a\; (Np^r)}(0) \left(\sum_{G'_N} \chi^{-1}(G'_N)\zeta_N^{G'_N}\right) \left(\sum_{G_{p^r}} \zeta_{p^r}^{aG_{p^r}}\right)\\
=& \chi(a)\zeta_{a\; (Np^r)}(0) \tau(\chi^{-1}) \sum_{G_{p^r}} \zeta_{p^r}^{aG_{p^r}}.
\end{split}
\]
\begin{itemize}
\item If $a$ is divisible by $p^i$ for some $0\leq i\leq r-1$, by Lemma~\ref{useful_formula_local_pairing}(5), one has $\sum_{G_{p^r}} \zeta_{p^r}^{aG_{p^r}}=0$.

\item If $a$ is divisible by $p^r$, namely, $a=ip^r$ for some $1\leq i\leq N-1$, then by Lemma~\ref{useful_formula_local_pairing}(1) and using the assumption that $\chi(p)=1$, one has 
$$
\chi(a)\zeta_{a\; (Np^r)}(0) \tau(\chi^{-1}) \sum_{G_{p^r}} \zeta_{p^r}^{aG_{p^r}}
=\phi(p^r) \tau(\chi^{-1}) \chi(i) \zeta_{i\; (N)}(0).
$$
\end{itemize}

It is known that the Dirichlet $L$-function $L(s,\chi)$ satisfies
$$
L(s,\chi)=\sum_{i=1}^{N-1} \chi(i) \zeta_{i\; (N)}(s)
$$
and has analytic continuation to the whole complex plane as $\chi$ is not a trivial character (for example, see \cite[Ch.~4]{Wa}). From the above discussion, one can simplify \eqref{eq:local_pairing_5} as
$$
\frac{\phi(p^r)(p-1)\log_p (\ell)}{p^r \phi(Np)} \tau(\chi^{-1}) \sum_{i\in \Z/N \Z} \chi(i) \zeta_{i\; (n)} (0)
= \frac{(p-1)\log (\ell)}{p \phi(N)} \tau(\chi^{-1}) L(0,\chi).\qedhere
$$
\end{proof}

\begin{cor}\label{cup_product_map}
We have
\begin{equation}\label{eq:cup_with_l}
(\ell ,1-\zeta_{Np^r})_{r,\theta,\mfr{p}}=\frac{(p-1)\log_p (\ell)}{p \phi(N)} \omega(N)\tau(\chi^{-1}) L(0,\chi)\in (\cO/p^r\cO)(1) 
\end{equation}
for all $\ell |N$ and
\begin{equation}\label{eq:cup_product_with_p}
(p,1-\zeta_{Np^r})_{r,\theta,\mfr{p}}=\frac{(p-1)}{p \phi(N)} \omega(N)\tau(\chi^{-1})\cL(\chi) L(0,\chi)\in (\cO/p^r\cO)(1).
\end{equation}
In particular, for all prime $q|Np$, $(q,1-\zeta_{Np^r})_{r,\theta,\mfr{p}}$ is non-zero for all $r$ big enough, and hence, the map $\varpi_{\theta,\mfr{p}}$ is surjective. 
\end{cor}

\begin{proof}
 Let $p_{r,N}$ be the inverse of $p^r$ in $(\Z/N\Z)^{\times}$. For each $r\in \Z_{\geq 1}$, we write $\zeta_{Np^r}=\exp(2\pi i/Np^r)$ and $\zeta_N^{p^{-r}}\zeta_{p^r}=\exp(2\pi i p_{r,N}/N+2\pi i /p^r)$. Then, for each $\ell |N$, one has
$$
(\ell ,1-\zeta_{Np^r})_{r,\theta,\mfr{p}}
=(\ell ,1-(\zeta_N^{p^{-r}}\zeta_{p^r})^{1/(p_{r,N}p^r+N)})_{r,\theta,\mfr{p}}
=\theta(p_{r,N}p^r+N)^{-1}(\ell ,1-(\zeta_N^{p^{-r}}\zeta_{p^r}))_{r,\theta,\mfr{p}}.
$$
Note that one can simplify $\theta(p_{r,N}p^r+N)^{-1}$ as $\omega(N)$. Thus, \eqref{eq:cup_with_l} follows from Theorem~\ref{local_pairing}. \eqref{eq:cup_product_with_p} is obtained by \eqref{eq:image_of_varpi} and by a consequence of \cite[Proposition~5.7]{BDP}, which asserts that for all prime $\ell$ divides $N$, one has
$$
(U^*_p-1)/(U^*_{\ell}-1)=\mathcal{L}(\chi)\log_p(\ell)^{-1}\in \mfr{h}^*_{\theta,\mfr{p}}/I^*_{\theta,\mfr{p}},
$$
where $\mathcal{L}(\chi)\neq 0$ is the $\mathcal{L}$-invariant attached to $\chi$.
\end{proof}


\section{Galois representations attached to cohomology of modular curves}\label{sec:05}
Throughout this section, we set $\mfr{h}^*:=\mfr{h}^*_{\theta,\mfr{p}}$ and $I^*:=I^*_{\theta,\mfr{p}}$ for simplicity. Let $\Lambda^{\#}_{\theta}$ be $\Lambda_{\theta}$ on which $G_{\Q}$ acts as follows. For $\sigma\in G_{\Q}$, $\sigma$ acts on $\Lambda_{\theta}^{\#}$ via the multiplication by $\theta^{-1}(\sigma) \langle \sigma \rangle^{-1}$, where $\<\sigma\>:=\<a\>\in \Lambda_{\theta}$ for some $a\in  1+p\Z_p$ satisfying $\sigma(\zeta_{p^r})=\zeta_{p^r}^a$ for all $r\geq 1$. 
Recall that we have seen in Section~\ref{sec:32} that $\mfr{h}^*=\Lambda_{\theta,\mfr{p}}$, $I^*=\mfr{p}$, and both $H_{\theta,\mfr{p}}^+$ and $H_{\theta,\mfr{p}}^-$ are free $\mfr{h}^*$-modules of rank $1$. 

Recall that $\mfr{h}^*$ is isomorphic to $\cT^{\cusp}$ in \cite{BDP}. By Theorem~3.5(i) in \textit{loc.~cit.}, the ring $\mfr{h}^*$ can be viewed as a certain universal deformation ring whose tangent space was computed in Proposition~2.7 of \textit{loc.~cit.} One can see from the computation of the quoted proposition that the trace of the Galois representation attached to $H_{\theta,\mfr{p}}^+\oplus H_{\theta,\mfr{p}}^-=H_{\theta,\mfr{p}}$ is reducible modulo $(I^*)^2$ since the trace of a Galois representation does not depend on the choice of the lattice. Therefore, by fixing a basis of $H_{\theta,\mfr{p}}^+\oplus H_{\theta,\mfr{p}}^-$ over $\mfr{h}^*$, this Galois representation modulo $I^*$ can be realized as either an upper triangular matrix, a lower triangular matrix, or a diagonal matrix. The goal of this section is to determine which one is the case. 

Since the Galois representation attached to $H_{\theta,\mfr{p}}^+/I^* H^+_{\theta,\mfr{p}} \oplus H_{\theta,\mfr{p}}^-/I^*H^-_{\theta,\mfr{p}}$ is reducible, at least one of $H_{\theta,\mfr{p}}^+/I^* H^+_{\theta,\mfr{p}}$ and $H_{\theta,\mfr{p}}^-/I^*H^-_{\theta,\mfr{p}}$ is stable under the action of $G_{\Q}$. We start by showing that the minus part is  $G_{\Q}$-stable. Let
$$
(\cdot,\cdot)_{\Lambda_{\theta}}:H_{\theta}\times H_{\theta} \to \Lambda_{\theta}
$$
be the perfect pairing constructed by Ohta satisfying
\begin{equation}
(\sigma x, \sigma y)=\kappa(\sigma)^{-1}\theta^{-1}(\sigma)\<\sigma\>^{-1}(x,y)
\end{equation}
for all $\sigma\in G_{\Q}$ (see \cite[Section~1.6.3]{FK} for the definition). Recall that we have shown in Corollary~\ref{image_of_varpi} that $\xi_{\theta}\{0,\infty\}_{\theta,DM}$ is in $H^-_{\theta}$. By the same argument as in Section~6.3.8 of \textit{loc.~cit.}, the above pairing induces a $\Lambda_{\theta}[G_{\Q}]$-module homomorphism
$$
(\cdot,\xi_{\theta}\{0,\infty\}_{\theta,DM})_{\Lambda_{\theta}}: H_{\theta}/I^*_{\theta} H_{\theta}\to \Lambda^{\#}_{\theta}/(\xi_{\theta}).
$$
Let $\mathcal{P}$ be the kernel of the above homomorphism after localizing at $\mfr{p}$, and let 
$$
Q:=(H_{\theta,\mfr{p}}/I^*H_{\theta,\mfr{p}})/\mathcal{P}\isom \Lambda^{\#}_{\theta,\mfr{p}}/(\xi_{\theta,\mfr{p}})\isom\Lambda^{\#}_{\theta,\mfr{p}}/\mfr{p}
$$
be the quotient. Indeed, one has $\mathcal{P}\isom (H^-_{\theta,\mfr{p}}/I^*H_{\theta,\mfr{p}}^-)(-1)$ and $Q\isom H^+_{\theta,\mfr{p}}/I^*H^+_{\theta,\mfr{p}}$ as $k_{\theta,\mfr{p}}[G_{\Q}]$-modules, since $\xi_{\theta,\mfr{p}}\{0,\infty\}_{\theta,DM,\mfr{p}}$ is a basis of $H^-_{\theta,\mfr{p}}/I^*H^-_{\theta,\mfr{p}}$ over $k_{\theta,\mfr{p}}$. Thus, we obtain a short exact sequence of $k_{\theta,\mfr{p}}[G_{\Q}]$-modules
\begin{equation}\label{eq:G_Q_lattice}
0\to H_{\theta,\mfr{p}}^-/I^*H_{\theta,\mfr{p}}^-
\to H_{\theta,\mfr{p}}/I^*H_{\theta,\mfr{p}}
\to H_{\theta,\mfr{p}}^+/I^*H_{\theta,\mfr{p}}^+
\to 0.
\end{equation}

As in \cite[Section 9.6]{FK}, we have an exact sequence of $k_{\theta,\mfr{p}}[G_{\Q}]$-modules
\begin{equation}\label{cuspidal extension}
0\rightarrow Q\rightarrow \frac{\widetilde{H}_{\theta,DM,\mfr{p}}}{(\ker: H_{\theta,\mfr{p}}\rightarrow Q)}\rightarrow \frac{\widetilde{H}_{\theta,DM,\mfr{p}}}{H_{\theta,\mfr{p}}}\rightarrow 0
\end{equation}
which gives an extension class in $H^1(\Z[\tfrac{1}{Np},\Lambda^{\#}_{\theta,\mfr{p}}/(\xi_{\theta,\mfr{p}})(1))$ by \eqref{eq:coho_exact_seq} and \eqref{eq:Drinfield_Manin1}.
Since $I^*=\mfr{p}$ is a principal ideal, one obtains from \eqref{cuspidal extension} by tensoring $I^*/(I^*)^2$ an exact sequence of $k_{\theta,\mfr{p}}[G_{\Q}]$-modules
\begin{equation}\label{exact sequence for c}
0\rightarrow  I^*H^+_{\theta,\mfr{p}}/{I^*}^2H^+_{\theta,\mfr{p}} \rightarrow (I^*H^+_{\theta,\mfr{p}}\oplus H^-_{\theta,\mfr{p}})/I^*(I^*H^+_{\theta,\mfr{p}}\oplus H^-_{\theta,\mfr{p}}) \to H^-_{\theta,\mfr{p}}/I^*H^-_{\theta,\mfr{p}} \rightarrow 0.    
\end{equation}

The following theorem describes the Galois representations attached to \eqref{eq:G_Q_lattice} and \eqref{exact sequence for c}.

\begin{thm}\label{split_gal_rep_at_p}
The short exact sequence \eqref{exact sequence for c}  does not split as $k_{\theta,\mfr{p}}[G_{\Q}]$-modules, and the short exact sequence \eqref{eq:G_Q_lattice} splits as $k_{\theta,\mfr{p}}[G_{\Q}]$-modules.
\end{thm}

\begin{proof}
To prove the first assertion, by the above discussion, it suffices to show that the extension class \eqref{cuspidal extension} is non-trivial. It is known \cite[Theorem~9.6.3]{FK} that the desired extension class coincides with the image of the family $(1-\zeta_{Np^r})_{r\geq 1,\theta,\mfr{p}}$ under the canonical homomorphism $\varprojlim_r (\Z[\tfrac{1}{Np},\zeta_{Np^r}]^{\times}\otimes \Z_p)_{\theta,\mfr{p}}\rightarrow H^1(\Z[\tfrac{1}{Np}],\Lambda^{\#}_{\theta,\mfr{p}}/(\xi_{\theta,\mfr{p}})(1))$ induced by the short exact sequence
\[
0\rightarrow \Lambda^{\#}_{\theta,\mfr{p}}(1)\rightarrow \Lambda^{\#}_{\theta,\mfr{p}}(1)\rightarrow \Lambda^{\#}_{\theta,\mfr{p}}/(\xi_{\theta,\mfr{p}})(1)\rightarrow 0
\]
and the fact that $\varprojlim_{r}(\Z[\tfrac{1}{Np},\zeta_{Np^r}]^{\times}\otimes\mathbb{Z}_p)_{\theta}\cong H^1(\Z[\tfrac{1}{Np}],\Lambda^{\#}_{\theta}(1))$. It is enough to prove that the image of $(1-\zeta_{Np^r})_{r\geq 1,\theta,\mfr{p}}$ in $H^1(\Q_p,\Lambda^{\#}_{\theta,\mfr{p}}/(\xi_{\theta,\mfr{p}})(1))$ is not trivial. Let $U_{\theta}$ be the $\theta$ eigenspace of the projective limit of semi-local units, and let $C_{\theta}$ be the $\Lambda_{\theta}$-submodule of $U_{\theta}$ generated by $(1-\zeta_{Np^r})_{r\geq 1,\theta}$. By \cite[Proposition ~5.2(a)(ii)]{tsuji} , we know that $U_{\theta,\p}/\p\cong H^1(\Q_{p},\Lambda^{\#}_{\theta,\mfr{p}}/(\xi_{\theta,\mfr{p}})(1))$. Note that both of these vector spaces are $2$-dimensional. By Theorem 3.1(2) in \textit{loc.~cit.}, the image of $C_{\theta}$ in $U_{\theta,\p}/\p$ is not trivial, which proves the first assertion.

To prove the second assertion, let $(e_-,e_+)$ be a basis of $H_{\theta,\mfr{p}}^- \oplus H_{\theta,\mfr{p}}^+$. Recall that the maximal ideal $\mfr{p}=I^*$ is generated by $f=T+1-\kappa(\gamma)\in \Lambda_{\theta,\mfr{p}}$. Then, it is clear that $(fe_+,e_-)$ is a basis of $I^*H^+_{\theta,\mfr{p}}\oplus H^-_{\theta,\mfr{p}}$. Thus, the second assertion follows from the first assertion as we have seen that the Galois representation attached to $I^*H^+_{\theta,\mfr{p}}\oplus H^-_{\theta,\mfr{p}}$ is reducible modulo $(I^*)^2$. 
\end{proof}
\bibliography{biblio}
\bibliographystyle{siam}
\end{document}